\documentclass[graybox]{svmult}

\usepackage{mathptmx}       
\usepackage{helvet}         
\usepackage{courier}        
\usepackage{type1cm}        
\usepackage{makeidx}         
\usepackage{graphicx}        
\usepackage{multicol}        
\usepackage[bottom]{footmisc}

\usepackage{amssymb}
\usepackage{amsmath}
\usepackage{ifthen}
\usepackage{amsfonts}
\usepackage{amscd}
\usepackage{amsxtra}
\usepackage{color}

\newcounter{alphabet}
\newcounter{tmp}
\newenvironment{Thm}[1][]{\refstepcounter{alphabet}%
\bigskip%
\noindent%
{\bf Theorem \Alph{alphabet}}%
\ifthenelse{\equal{#1}{}}{}{ (#1)}%
{\bf .} \itshape}{\vskip 8pt}

\makeatletter
\newcommand{\Ref}[1]{\@ifundefined{r@#1}{}{\setcounter{tmp}{\ref{#1}}\Alph{tmp}}}
\makeatother

\def\be{\begin{equation}}
\def\ee{\end{equation}}

\newcommand{\IC}{{\mathbb C}}
\newcommand{\ID}{{\mathbb D}}

\newcommand{\dist}{{\operatorname{dist}}}
\newcommand{\ba}{\begin{array}}
\newcommand{\ea}{\end{array}}
\newcommand{\beq}{\begin{eqnarray}}
\newcommand{\eeq}{\end{eqnarray}}
\newcommand{\beqq}{\begin{eqnarray*}}
\newcommand{\eeqq}{\end{eqnarray*}}
\newcommand{\br}{\begin{remark}}
\newcommand{\er}{\end{remark}}
\newcommand{\brs}{\begin{remarks}}
\newcommand{\ers}{\end{remarks}}
\newcommand{\bprob}{\begin{problem}}
\newcommand{\eprob}{\end{problem}}
\newcommand{\ds}{\displaystyle}

\makeindex             

\begin{document}

\title*{On the Bohr inequality}
\author{Y. Abu Muhanna,  R. M. Ali, and S. Ponnusamy
}
\institute{Yusuf  Abu Muhanna \at  Department of Mathematics,
American University of Sharjah, UAE-26666.
\email{ymuhanna@aus.edu}
\and  Rosihan M. Ali \at School of Mathematical Sciences, Universiti Sains Malaysia, 11800
USM Penang, Malaysia. \email{rosihan@usm.my}
\and Saminathan Ponnusamy \at
Indian Statistical Institute (ISI), Chennai Centre, SETS (Society
for Electronic Transactions and Security), MGR Knowledge City, CIT
Campus, Taramani, Chennai 600 113, India.
\email{samy@isichennai.res.in}}

%
%
\maketitle

\abstract{The Bohr inequality, first introduced by Harald Bohr in 1914, deals with finding the largest radius $r$, $0<r<1$, such that $\sum_{n=0}^\infty |a_n|r^n \leq 1$ holds whenever $|\sum_{n=0}^\infty a_nz^n|\leq 1$ in the unit disk $\mathbb{D}$ of the complex plane. The exact value of this largest radius, known as the \emph{Bohr radius}, has been established to be $1/3.$ This paper surveys recent advances and generalizations on the Bohr inequality. It discusses the Bohr radius for certain power series in $\mathbb{D},$ as well as for analytic functions from $\mathbb{D}$ into particular domains. These domains include the punctured unit disk, the exterior of the closed unit disk, and concave wedge-domains. The analogous Bohr radius is also studied for harmonic and starlike logharmonic mappings in $\mathbb{D}.$ The Bohr phenomenon which is described in terms of the Euclidean distance is further investigated using the spherical chordal metric and the hyperbolic metric. The exposition concludes with a discussion on the $n$-dimensional Bohr radius.\\[10pt]
Keywords: Power series; polynomials; Bohr inequality; Bohr radius; Bohr phenomenon; analytic functions;  Banach spaces;   operator inequality;
concave-wede domain; subordination; majorant series; alternating series; Dirichlet series;  harmonic mappings; holomorphic functions; Reinhardt domains; multidimensional Bohr radius.\\[10pt]
2010 Mathematics Subject Classification: 30A10; 30B10; 30B50;  30C35; 30C45; 30C80; 30H05;  32A05;  32A07; 32A10;  46L06;  47A56;  47A13.}

\section{Harald August Bohr (1887--1951)}\label{sec1}

Harald August Bohr was born on the twenty-second day
of April, 1887 in Copenhagen, Denmark, to Christian and Ellen Adler Bohr. His father was a distinguished professor of physiology
at the University of Copenhagen and his elder brother Niels was to become a famous theoretical physicist.

Harald and Niels were prolific football (soccer) players. Harald made his playing debut as a 16-year-old in 1903 with \emph{Akademisk Boldklub}. He represented the Danish national football team in the 1908 Summer Olympics, where football was first introduced as an official event.
Denmark faced hosts Great Britain in the final, but eventually lost 2-0, and Bohr and the Danish team came home as silver medalists.

Bohr enrolled at the University of Copenhagen in 1904 to study mathematics. It was reported that during his doctoral dissertation examination, there were more football fans in attendance than there were mathematicians!

Bohr became a professor in mathematics at the Copenhagen's Polytechnic Institute in 1915. He was later appointed as
professor at the University of Copenhagen in 1930, where he remained in that position until his demise on  January 22, 1951.

Bohr was an extremely capable teacher. Indeed to his honour, the annual award for outstanding teacher at the
University of Copenhagen is called \emph{The Harald}. With Johannes Mollerup, Bohr wrote
an influential four-volume textbook entitled \textit{Lærebog i Matematisk Analyse} (Textbook in mathematical analysis).


Bohr worked on Dirichlet series, and applied analysis to the theory of numbers. During this period, Edmund Landau was at G\"{o}ttingen, studying the Riemann zeta function $\zeta (s)$, and whom was also renowned for his unsolved problem on Landau's constant (see for example, \cite{samy-book3}).  Bohr  collaborated with Landau, and in 1914, they proved the Bohr-Landau theorem on the distribution of zeros for the zeta function. All but an infinitesimal proportion of these zeros lie in a small neighbourhood of the line $s = 1/2$. Although Niels Bohr was an accomplished physicist and Nobel Laureate, Harald and Niels only had one joint publication.

Bohr's interest in functions which could be represented by a Dirichlet series led to the development of
almost periodic functions. These are functions which, after a period,
take values within $e$ of the values in the previous period.
Bohr pioneered this theory and presented it in three major works during
the years 1923 and 1926 in \emph{Acta Mathematica}. It is with these works that
his name is now most closely associated.


Titchmarsh \cite{Titch} made the following citation on Bohr's work on almost periodic functions: ``\textit{The general theory was developed for the
case of a real variable, and then, in the light of it, was developed the most beautiful theory of almost
periodic functions of a complex variable.
The creation of the theory of almost periodic functions of a real
variable was a performance of extraordinary power, but was not based on the most up-to-date methods,
and the main results were soon simplified and improved. However, the theory of almost periodic functions
of a complex variable remains up to now in the same perfect form in which it was given by Bohr}''.

Bohr devoted his life to mathematics and to the theory of almost periodic functions. Four months before his death, Bohr was still actively engaged with the mathematical community at the International Congress of Mathematicians in Cambridge, Massachusetts,
in September, 1950; he died soon after the New Year.
Besicovitch wrote: ``\textit{For most of his life Bohr was a sick man. He used to suffer from bad headaches and had to avoid all mental effort.
Bohr the man was not less remarkable than Bohr the mathematician. He was a man of refined intellect, harmoniously
developed in many directions. He was also a most humane person. His help to his pupils, to his colleagues and friends,
and to refugees belonging to the academic world was generous indeed. Once he had decided to help he stopped at
nothing and he seldom failed. He was very sensitive to literature. His favourite author was Dickens; he had a
deep admiration of Dickens' love of the human being and deep appreciation of his humour}''.

Harald Bohr was elected an Honorary Member of the London Mathematical Society in 1939. Additional biographical account on Bohr may be obtained from http://en.wikipedia.org/wiki/Harald${}_{-}$Bohr

\section{The classical Bohr inequality}\label{sec2}
Let $\mathbb{D}$ denote the unit disk $\{z\in \IC:\, |z| < 1 \}$  and let $\mathcal{A} $ denote the space of functions analytic in
the unit disk $\mathbb{D}$. Then  $\mathcal{A} $ is a locally convex linear topological vector space
endowed with the topology of uniform convergence over compact subsets of $\mathbb{D}$. We may assume that $f\in {\mathcal A}$ has continuous boundary values
and $\|f\|_\infty =\sup_{z\in\ID}|f(z)|$. It is evident that each $f\in {\mathcal A}$  has a power series expansion about the origin. What can be deduced from the sum of the moduli of the terms in the series?

In 1914,  Harald Bohr \cite{Bohr-14} studied this property and made the observation: ``In particular, the solution of what is called the
``absolute convergence problem'' for Dirichlet series of the type
$\sum a_n n^{-s}$ must be based upon a study of the relations between
the absolute value of a power-series in an infinite number of
variables on the one hand, and the sum of the absolute values
of the individual terms on the other. It was in the course of this
investigation that I was led to consider a problem concerning
power-series in one variable only, which we discussed last year,
and which seems to be of some interest in itself."
More precisely, Bohr obtained the following remarkable result.

\begin{Thm}\label{Bohr-TheoA}
{\rm ([Bohr inequality (1914)])}
Let $f(z)=\sum\limits_{n=0}^{\infty }a_{n}z^{n}$ be analytic in $\mathbb{D}$, and $\|f\|_\infty :=\sup_{z\in\ID}|f(z)|<\infty$
for all $z\in \mathbb{D}.$ Then
\begin{equation}\label{eq1.1}
M_f(r):=\sum\limits_{n=0}^{\infty }|a_{n}|r^{n} \leq \|f\|_\infty
\end{equation}
for $0\leq r\leq 1/3.$
\end{Thm}

Here $M_f$ is the associated majorant series for $f$. Bohr actually obtained the inequality \eqref{eq1.1}  only for $r\leq {1}/{6}.$ M.~Riesz, I.~Schur and N.~ Wiener independently proved
its validity for $r\leq {1}/{3}$ and showed that the bound ${1}/{3}$
was sharp. The best constant $r$ in \eqref{eq1.1}, which is $1/3,$ is called the Bohr radius
for the class of all analytic self-maps of the unit disk $\ID$. Other proofs can also be found, for example, by Sidon \cite{Sidon-27-15},  Tomic \cite{Tomic-62-16}
and Paulsen {\it et al.} in  \cite{PaulPopeSingh-02-10} and \cite{PaulSingh-04-11, PaulSingh-06-12}. Similar problems were considered for Hardy spaces or for more abstract spaces, for instance, by Boas and Khavinson in \cite{BoasKhavin-97-4}. More recently, Aizenberg \cite{Aizen-00-1,Aizen-05-3} extended the inequality in the context of several complex variables which we shall discuss with some details in Section \ref{sec1-ndim}.

\subsection{Bohr phenomenon for the space of subordinate mappings}\label{sec2.1}
In recent years, two types of spaces are widely considered in the study of Bohr inequality. They are the space of subordinations and
the space of complex-valued bounded harmonic mappings. One way of generalizing the notion of the Bohr phenomenon, initially defined for
mappings from $\mathbb{D}$  to itself, is to rewrite Bohr inequality in the equivalent form
$$\sum_{k=1}^\infty |a_k|r^k\le 1-|a_0|=1-|f(0)|.
$$
The distance to the boundary is an important geometric quantity. Observe that the number $1-|f(0)|$ is the distance from the point
$f(0)$ to the boundary $\partial \mathbb{D}$ of the unit disk $\mathbb{D}.$ Using this ``distance form" formulation of Bohr inequality, the notion of  Bohr radius can be generalized to the  class of functions
$f$ analytic in $\mathbb{D}$ which take values in a given domain $\Omega$. For our formulation, we first introduce the
notion of subordination.

If $f$ and $g$ are analytic in $\ID$, then $g$ is {\it
subordinate} to $f$, written $g\prec f$ or $g(z)\prec f(z)$, if there exists a
function $w$ analytic in $\ID$ satisfying $w(0)=0$, $|w(z)|<1$ and
$g(z)=f(w(z))$ for $z\in \ID$.  If $f$ is univalent in $\mathbb{D}$, then $g\prec f$ if and only
if $g(0)=f(0)$ and $g(\ID )\subset f(\ID )$ (see \cite[ p.~190 and p.~253]{DurenUniv-83-8}). By the Schwarz lemma, it follows that
$$|g'(0)|= |f'(w(0))w'(0)| \leq |f'(0)|.
$$

Now for a given $f,$ let $S(f)=\{g:\, g\prec f\}$ and $\Omega =f(\ID)$. The family $S(f)$ is said to satisfy a Bohr phenomenon if there exists an $r_{f}$, $0<r_{f}\leq 1$ such that whenever $g(z)=\sum_{n=0}^{\infty} b_nz^n\in S(f)$, then
\begin{equation}\label{sub}
\sum_{n=1}^{\infty} |b_n|r^n =M_g(r)-|b_0|\leq \dist (f(0),\partial f(\ID))
\end{equation}
for $|z|=r<r_f.$
We observe that if  $f(z)=(a_0-z)/(1-\overline{a_0}z)$ with $|a_0|<1$, and $\Omega =\ID $, then $\dist (f(0),\partial \Omega)=1-|a_0|=1-|b_0|$
so that \eqref{sub} holds with $r_f=1/3$, according to Theorem \Ref{Bohr-TheoA}. We say that the family $S(f)$ satisfies the \textit{classical}
Bohr phenomenon if \eqref{sub} holds for $|z|=r<r_0$ with $1-|f(0)|$ in place of $\dist (f(0),\partial f(\ID))$.
Hence the distance form allows us to extend Bohr's theorem to a variety of distances provided the Bohr phenomenon exists. 
The following theorem was obtained in \cite[Theorem 1]{Abu}:

\begin{theorem}\label{subtheo}
If $f,g$ are analytic in $\ID$ such that $f$ is univalent in $\ID$ and $g\in S(f)$, then inequality \eqref{sub}
 holds with $r_f=3-2\sqrt{2}\approx 0.17157$. The sharpness of $r_f$ is shown by the Koebe function $f(z)=z/(1-z)^2.$
\end{theorem}
\begin{proof}
Let $g(z)=\sum_{n=0}^{\infty} b_nz^n\prec f(z)$, where  $f$ is a univalent mapping of $\ID$ onto a simply connected domain $\Omega =f(\ID)$.
Then it is well known that (see, for instance, \cite[p.~196]{DurenUniv-83-8} and \cite{DeB1})
\be\label{eq1-subtheo}
\frac{1}{4}|f'(0)|\leq \dist (f(0),\partial \Omega )\leq |f'(0)| ~\mbox{ and }~ |b_n| \leq n |f'(0)|.
\ee
It follows that $|b_n| \leq 4n \dist (f(0),\partial \Omega ),$ and thus
\begin{eqnarray*}
\sum_{n=1}^{\infty} |b_n|r^n \leq 4\dist (f(0),\partial \Omega )  \sum_{n=1}^{\infty} nr^n =4 \dist (f(0),\partial \Omega )  \frac{r}{(1-r)^2}
\leq \dist (f(0),\partial \Omega )
\end{eqnarray*}
provided $4 r\leq (1-r)^2 $, that is, for $r\leq 3-2\sqrt{2}.$ When $f(z)=z/(1-z)^{2}$, we obtain
$\dist (f(0),\partial \Omega )=1/4$ and a simple calculation gives sharpness.
\qed \end{proof}

In \cite{Abu}, it was also pointed out that for $f(z)=z/(1-z)^2$, $S(f)$ does not have the classical Bohr phenomenon. Moreover, from the proof of Theorem \ref{subtheo}, it is easy to see that $r_f=3-2\sqrt{2}$ could be replaced by $1/3$ if $f$ is univalent in $\ID$ with convex image.
In this case, instead of \eqref{eq1-subtheo}, one uses the following (see  \cite[p.~195, Theorem 6.4]{DurenUniv-83-8}):
$$\frac{1}{2}|f'(0)|\leq \dist (f(0),\partial \Omega )\leq |f'(0)| ~\mbox{ and }~ |b_n| \leq  |f'(0)|.
$$
Hence the deduced result clearly contains the classical Bohr inequality \eqref{eq1.1}. In the next section, we shall present various other generalizations and improved results.

\subsection{Bohr radius for alternating series and symmetric mappings}
The majorant series of $f$ defined by \eqref{eq1.1} belongs to a very important class
of series, namely,  series with non-negative terms. As pointed out by the authors in  \cite{AliBarSoly}, there is yet another class of interesting series--the class of
alternating series. Thus, for $f(z)=\sum\limits_{n=0}^{\infty }a_{n}z^{n},$ its associated
\emph{alternating series} is given by
\begin{equation} \label{0.6}
A_f(r)=\sum_{k=0}^\infty (-1)^k|a_k|\,r^k.
\end{equation}
In \cite{AliBarSoly}, the authors obtained several results on Bohr radius, which include the following counterpart of Theorem \Ref{Bohr-TheoA}.

\begin{theorem} \label{Bohr-Theorem-2}
If $f(z)=\sum\limits_{n=0}^{\infty }a_{n}z^{n}$ is analytic and bounded in $\mathbb{D},$ then
$\left|A_f(r)\right| \leq \|f\|_\infty$
for $0\leq r\leq 1/\sqrt{3}$. The radius $r=1/\sqrt{3}$ is best possible.
\end{theorem}

In \cite{AliBarSoly}, the authors proved the following.
\begin{theorem} \label{Lemma-2}
Let  $f(z)=\sum_{k=1}^\infty a_{2k-1} z^{2k-1}$ be an odd analytic
function in $\mathbb{D}$ such that $|f(z)|\le 1$ in $\mathbb{D}$.
Then $ M_f(r) \le 1$ for $0\leq r\leq r_*$, where $r_*$ is a solution of the equation %
$$ 5r^4+4r^3-2r^2-4r+1=0,
$$
which is unique in the interval $1/\sqrt{3}<r<1$. The value of $r_*$ can be calculated in terms of radicals as
\begin{equation*} 
r_*=-\frac{1}{5}+\frac{1}{10}\sqrt{\frac{A+32}{3}}+
\frac{1}{10}\sqrt{\frac{64}{3}-\frac{A}{3}+144\sqrt{\frac{3}{A+32}}}=
0.7313\ldots,
\end{equation*}
where
$$ A=10\cdot 2^{\frac{2}{3}}\left((47-3\sqrt{93})^{\frac{1}{3}}+(47+3\sqrt{93})^{\frac{1}{3}}\right).
$$
\end{theorem}

In \cite{AliBarSoly}, an example was given  to conclude that the
Bohr radius for the class of odd functions satisfies the
inequalities $r_*\le r\le r^*\approx 0.7899 $, where
$$ r^*= \frac{1}{4}\sqrt{\frac{B-2}{6}}+\frac{1}{2}
\sqrt{3\sqrt{\frac{6}{B-2}}-\frac{B}{24}-\frac{1}{6}},
$$
with
$$B=(3601-192\sqrt{327})^{\frac{1}{3}}+(3601+192\sqrt{327})^{\frac{1}{3}}.
$$
This raises the following open problem.

\bprob (\cite{AliBarSoly})
{\it Find the Bohr radius for the class of odd functions $f$ satisfying $|f(z)|\le 1$ for all $z\in
\mathbb{D}$.}
\eprob

Apart from the majorant and alternating series defined by \eqref{eq1.1} and \eqref{0.6}, one can
consider a more general type of series associated with $f$ given by
\begin{equation} \label{2.203}
S_f^n(r)=\sum_{k=0}^\infty e^{\frac{2\pi i k}{n}}|a_k|r^k,
\end{equation}
where $n$ is a positive integer. Note that
$$M_f(r)=S_f^1(r) ~\mbox{ and }~  A_f(r)=S_f^2(r).
$$
The arguments of coefficients of series \eqref{2.203} are equally
spaced over the interval $[0,2\pi),$ and thus $S_f^n(r)$ can be
thought of  as an \emph{argument symmetric series}
associated with $f$.  This raised the next problem.

\bprob (\cite{AliBarSoly})
Given a positive integer $n\ge 2,$ and $|f(z)|\le 1$ in $\mathbb{D},$ find the largest radius $r_n$ such that $|S_f^n(r)|\le 1$ for all $r\le
r_n.$
\eprob

We recall that an analytic function in $\mathbb{D}$ is called $n$-symmetric,
where $n\ge 1$ is an integer, if $f(e^{2\pi i/n}z)=f(z)$ for all
$z\in \mathbb{D}$. It is a simple exercise to see that $f(z)=\sum\limits_{n=0}^{\infty }a_{n}z^{n}$ is $n$-symmetric if
and only if its Taylor expansion has the $n$-symmetric form
$$ 
f(z)=a_0+\sum_{k=1}^\infty a_{nk}z^{nk}.
$$ 
In \cite{AliBarSoly}, the authors generalized Theorem \Ref{Bohr-TheoA} as follows.

\begin{theorem} \label{Lemma-1} %
If $f(z)=\sum_{k=0}^\infty a_{nk} z^{nk}$ is analytic in $\mathbb{D}$,
and  $\left|f(z)\right|\le 1$ in $\mathbb{D}$, then $M_f(r)\le 1$
for $0\leq r\leq 1/\sqrt[n]{3}$.  The radius $r=1/\sqrt[n]{3}$ is best possible.
\end{theorem} %
\begin{proof} Put $\zeta =z^n$ and consider a function
$g(\zeta)=\sum_{k=0}^\infty a_{nk}\zeta^k$. Clearly, $g$ is analytic in $\ID$ and
$|g(\zeta)|=|f(z)|\le 1$ for all $|\zeta|<1$. Thus, $|a_{nk}|\leq 1-|a_0|^2$ for all $k\geq 1$
and this well known inequality is easily established (see \cite{Bohr-14} and \cite[Exercise 8, p.172]{Nehari}).
For $r^n\leq 1/3$,
\beqq
M_f(r)&\leq &  |a_0| +( 1-|a_0|^2)\sum_{k=1}^\infty r^{nk}=|a_0| +( 1-|a_0|^2)\frac{r^n}{1-r^n}\\
&\leq& |a_0| +( 1-|a_0|^2)\frac{1/3}{1-(1/3)} :=\frac{1}{2}h(|a_0|),
\eeqq
where $h(x)=1+2x-x^2$, $0\leq x<1$. Since $h(x)\leq h(1)$, it follows that
$$M_f(r)=\sum_{k=0}^\infty |a_{nk}| r^{nk} \leq 1 ~\mbox{ for $r\leq 1/\sqrt[n]{3}$.}
$$

To show that the radius $1/\sqrt[n]{3}$ is best possible, consider
$$\varphi_\alpha (\zeta)=\frac{\alpha-\zeta}{1-\overline{\alpha}\zeta }
$$
with $\zeta=z^n.$ Then, for each fixed $\alpha \in \ID$, $\varphi_\alpha $ is analytic in $\ID$, $\varphi_\alpha (\ID)=\ID$ and $\varphi_\alpha (\partial \ID)=\partial\ID.$ It suffices to restrict  $\alpha$ such that $0<\alpha<1$. Moreover, for $|\zeta|<1/\alpha $,
$$\varphi_\alpha (\zeta)=(\alpha-\zeta)\sum_{k=0}^\infty \alpha^{k} \zeta^{k}=\alpha - (1-\alpha^2)\sum_{k=1}^\infty \alpha^{k-1} \zeta^{k},
$$
and with $|\zeta|=\rho,$
$$M_{\varphi_\alpha} (\rho)=\alpha + (1-\alpha^2)\sum_{k=1}^\infty \alpha^{k-1} \rho ^{k}=2\alpha -\varphi_\alpha (\rho ).
$$
It follows that $M_{\varphi_\alpha} (\rho)>1$ if and only if $(1-\alpha)((1+2\alpha)\rho -1)>0$, which gives $\rho >1/(1+2\alpha)$. Since $\alpha$ can be chosen
arbitrarily close to $1,$ this means that the radius $r=1/\sqrt[n]{3}$ in Theorem \ref{Lemma-1}  is best possible.
\qed\end{proof}

%
In Theorem \ref{Lemma-1}, it would be interesting to find the smallest constant for Bohr inequality to hold when $r>1/\sqrt[n]{3}$.
From the proof, it is clear that
\beqq
M(r) &=&  \sup\left \{ t +( 1-t^2)\frac{r^n}{1-r^n}:\, 0\leq t=|a_0|\leq 1\right \}\\
&=&\left \{\begin{array}{cl}
1& \mbox{ for }~ 0\leq r\leq 1/\sqrt[n]{3} \\
\ds \frac{4r^{2n}+(1-r^n)^2}{4r^{n}(1-r^n)} & \mbox{ for }~ 1/\sqrt[n]{3}<r<1.
\end{array}\right .
\eeqq
On the other hand, it follows from the argument of Landau, which is an immediate consequence of the Cauchy-Bunyakovskii inequality, that
$$M_f(r)\leq \left (\sum_{k=0}^\infty |a_{nk}|^2\right )^{1/2} \left (\sum_{k=0}^\infty r^{2nk} \right )^{1/2}= \frac{\|f\|_2}{\sqrt{1-r^{2n}}} \leq \frac{\|f\|_\infty}{\sqrt{1-r^{2n}}},
$$
where $\|f\|_2$ stands for the norm of the Hardy space $H^2(\ID)$. Thus the following result is obtained.

\begin{corollary} \label{cor-1a} %
If $f(z)=\sum_{k=0}^\infty a_{nk} z^{nk}$ is bounded and analytic in $\mathbb{D}$, then
$$M_f(r)\le A_n(r)\|f\|_\infty,
$$
where $A_n(r)=\inf\{M(r), 1/\sqrt{1-r^{2n}}\}. $
\end{corollary} %

We remark that for $r$ close to $1$, $M(r)\sqrt{1-r^{2n}}>1$ which is reversed for $r$ close to $1/\sqrt[n]{3}$.
So a natural question is to look for the best such constant $A(r)$.  In  \cite{BombBour-196},  Bombieri determined
the exact value of this constant for the case $n=1$ and for $r$ in the range $1/3\leq  r\leq 1/\sqrt{2}.$ This constant is
$$A(r ) =\frac{3-\sqrt{8(1-r^2)}}{r}.
$$
Later  Bombieri and Bourgain in \cite{BombBour-2004} considered the function
$$m(r)=\sup \left \{\frac{M_f(r)}{\|f\|_\infty} \right \}
$$
for the case $n=1$, and studied the behaviour of $m(r)$  as $r\rightarrow 1$ (see also \cite{DjakovRama-2000}). More precisely,
the authors in \cite[Theorem 1]{BombBour-2004} proved the following
result which validated a question raised in \cite[Remark 1]{PaulPopeSingh-02-10} in the affirmative.

\begin{theorem} \label{theo-BB}
If $r>1/\sqrt{2}$, then $m(r)< 1/\sqrt{1-r^2}$. With $\alpha =1/\sqrt{2}$, the function $\varphi_\alpha (z)=(\alpha-z)/(1-\alpha z)$  is extremal giving
$m(1/\sqrt{2})=\sqrt{2}$.
\end{theorem}

A lower estimate for $m(r)$ as $r\rightarrow 1$ is also obtained in \cite[Theorem 2]{BombBour-2004}. Given $\epsilon >0$, there exists a positive constant $C(\epsilon )>0$ such that
$$\frac{1}{1-r^2} -C(\epsilon ) \left ( \log \frac{1}{1-r}\right )^{\frac{3}{2}+\epsilon}\leq m(r)
$$
as $r\rightarrow 1$.
A multidimensional generalization of the work in \cite{BombBour-2004} along with several other issues, including on the Rogosinski phenomena,
is discussed in a recent article by Aizerberg \cite{Aizen-12}. More precisely, the following problems were treated in \cite{Aizen-12} (see
also Section \ref{sec1-ndim}):
\begin{enumerate}
\item Asymptotics of the majorant function in the Reinhardt domains in ${\mathbb C}^n$.
\item The Bohr and Rogosinski radii for Hardy classes of functions holomorphic in the
disk.
\item Neither Bohr nor Rogosinski radius exists for functions holomorphic in an annulus
with natural basis.
\item  The Bohr and Rogosinski radii for the mappings of the Reinhardt domains into
Reinhardt domains.
\end{enumerate}

If $a_0=0$, it follows from the proof of Theorem \ref{Lemma-1} that the number $r=1/\sqrt[n]{3}$ in Theorem \ref{Lemma-1} can be evidently replaced
by $r=1/\sqrt[n]{2}$.

\begin{corollary} \label{cor1} %
If $f(z)=\sum_{k=1}^\infty a_{nk} z^{nk}$ is analytic in $\mathbb{D}$,
and  $\left|f(z)\right|\le 1$ in $\mathbb{D}$, then $M_f(r)\le 1$
for $0\leq r\leq 1/\sqrt[n]{2}$.  The radius $r=1/\sqrt[n]{2}$ is best possible as demonstrated by the function
$$\varphi_\alpha (z)=z^n\left (\frac{\alpha-z^n}{1-\alpha z^n}\right )
$$
with $\alpha =1/\sqrt[n]{2}.$
\end{corollary}

We now state another simple extension of Theorem \ref{Lemma-1} which again contains the classical Bohr inequality for the special case $n=1$.

\begin{theorem} \label{Lemma-1a}
If $f(z)=\sum_{k=0}^\infty a_{nk} z^{nk}$ is analytic in $\mathbb{D}$
satisfying ${\rm Re}\, f(z)\leq  1$ in $\mathbb{D}$ and $f(0)=a_0$ is positive, then $M_f(r)\le 1$
for $0\leq r\leq 1/\sqrt[n]{3}$.
\end{theorem}
\begin{proof}
The proof requires the well known coefficient inequality for functions with positive real part. If $p(z)=\sum_{k=0}^\infty p_{k} z^{k}$ is 
analytic in $\ID$ such that ${\rm Re}\, p(z)> 0$ in $\ID$, then $|p_k|\leq 2{\rm Re}\,p_0$
for all $k\geq 1$. Applying this result to $p(z)=1-f(z)$ leads to $|a_{nk}|\leq 2(1-a_0)$ for all $k\geq 1$.  Thus
$$M_f(r)\leq a_0 +2(1-a_0)\sum_{k=1}^\infty r^{nk}=a_0 +2(1-a_0)\frac{r^n}{1-r^n},
$$
which is clearly less than or equal to $1$ if  $r^n\leq 1/3$.
\qed\end{proof}

A minor change in the proof of Theorem \ref{Lemma-1}   gives the following result, which for $a_0=0$ provides a vast improvement
on the Bohr radius.

\begin{corollary} \label{cor2}  %
If $f(z)=\sum_{k=0}^\infty a_{nk} z^{nk}$ is analytic in $\mathbb{D}$,
and  $\left|f(z)\right|\le 1$ in $\mathbb{D}$, then
$$ |a_0|^2 +\sum_{k=1}^\infty |a_{nk}|r^{nk}\le 1
$$
for $0\leq r\leq 1/\sqrt[n]{2}$.  The radius $r=1/\sqrt[n]{2}$ is best possible.
\end{corollary}
\begin{proof}
As in the proof of Theorem \ref{Lemma-1}, it follows easily that for $r^n\leq 1/2$,
\beqq
|a_0|^2 +\sum_{k=1}^\infty |a_{nk}|r^{nk} &\leq &  |a_0|^2 +(1-|a_0|^2)\frac{r^n}{1-r^n}\leq |a_0|^2 +(1-|a_0|^2)=1.
\eeqq
Also, it is easy to see that this inequality fails to hold for larger $r$.
\qed\end{proof}

The case $n=1$ of Theorem \ref{Lemma-1a} and Corollary \ref{cor2} appeared in \cite{PaulPopeSingh-02-10}.

\subsection{Bohr phenomenon for harmonic mappings}

Suppose that $f=u+iv$ is a complex-valued harmonic function defined on a simply connected domain $D.$ Then $f$ has the canonical form $f=h+\overline{g}$, where $h$ and $g$ are analytic in $D$. A generalization of Bohr inequality for harmonic functions from $\ID$ into $\ID$ was initiated by Abu Muhanna in \cite{Abu}.

\begin{theorem}\label{Muhan-CVth2}
Let $f(z)=h(z)+\overline{g(z)}=\sum_{n=0}^{\infty}a_nz^n+\sum_{n=1}^{\infty}\overline{b_nz^n}$ be a complex-valued harmonic function in $\ID$.
If $|f(z)|<1$ in $\ID$, then
$$
\sum_{n=1}^{\infty}(|a_n|+|b_n|)r^n\leq \frac{2}{\pi} \approx 0.63662
$$
and
\begin{equation}\label{Muhhar}
\sum_{n=1}^{\infty}|\,e^{i\mu}a_n+e^{-i\mu}b_n| r^n+|{\rm Re}\,e^{i\mu}a_0|\leq 1
\end{equation}
for $r\leq 1/3$ and any real $\mu$. Equality in \eqref{Muhhar} is attained by the M\"{o}bius transformation
$$ \varphi(z)= \frac{z-a}{1-az} ,\quad 0<a<1, \quad\text{as $a\rightarrow 1$.}
$$
\end{theorem}

From the proof of Theorem \ref{subtheo}, it suffices to have   sharp upper estimates for $|a_n|+|b_n|$ and $|\,e^{i\mu}a_n+e^{-i\mu}b_n|.$ With the help of these estimates (see \cite[Lemma 4]{Abu} and \cite {CPW-BMMSS2011,CPW-JMMA2011}),
the proof of Theorem \ref{Muhan-CVth2} is readily established. In \cite{Abu}, an example of a harmonic function was given to show that the inequality \eqref{Muhhar} fails
when $|{\rm Re\,}e^{i\mu }a_{0}|$ is replaced by $|a_{0}|.$

Theorem \ref{Muhan-CVth2} was extended to bounded domains
in \cite{Abu3}. If $D$ is a bounded set, denote by $\overline{D}$ the closure of $D$, and $\overline{D}_{{\rm min}}$ the smallest closed disk containing the closure of $D$.

\begin{theorem}
{\rm (\cite[Theorem 4.4]{Abu3})}
Let $f(z)=h(z)+\overline{g(z)}=\sum_{n=0}^{\infty}a_nz^n+\sum_{n=1}^{\infty}\overline{b_nz^n}$ be a complex-valued harmonic function in $\ID$.
If $f:\,\ID\rightarrow D$ for some bounded domain $D$, then, for $r\leq 1/3$,
$$\sum_{n=1}^{\infty}(|a_n|+|b_n|)r^n\leq \frac{2}{\pi}\rho
$$
and
$$\sum_{n=1}^{\infty}|e^{i\mu}a_n+e^{-i\mu}b_n|\,r^n+|{\rm Re}\,e^{i\mu}(a_0-w_0)|\leq \rho,
$$
where $\rho$ and $w_0$ are, respectively, the radius and centre of $\overline{D}_{{\rm min}}$.

The bound $1/3$ is sharp as demonstrated by an analytic univalent mapping $f$ from $\ID$ onto $D$. In particular, if $D$ is an open disk
with radius $\rho>0$ centred at $\rho w_0$, then sharpness is shown by the M\"{o}bius transformation
$$
\varphi(z)=e^{i\mu_0}\rho\left(\frac{z+a}{1+az}+|w_0|\right)
$$
for some $0< a<1$ and $\mu_0$ satisfying $w_0=|w_0|e^{i\mu_0}$.
\end{theorem}

\subsection{Bohr inequality in hyperbolic metric}

In \cite{Abu2}, Abu Muhanna and Ali expressed the Bohr inequality in terms of the spherical chordal distance
$$
\chi (z_1,z_2)=\frac{|z_1-z_2|}{\sqrt{1+|z_1|^2}\sqrt{1+|z_2|^2}}, \quad z_1,z_2\in \mathbb{C}.
$$
Thus the Euclidean distance in inequality \eqref{sub} is replaced by the chordal distance $\chi$.

Let $c\overline{\ID }$ denote the complement of $\ID \cup \partial \ID $ and
${\mathcal H}(\ID,\Omega)$ be the class consisting of all analytic functions mapping $\ID$ into $\Omega$.
Denote by ${\mathcal H}(\ID):={\mathcal H}(\ID,\ID)$.
The following theorem generalizes Bohr's theorem for the class ${\mathcal H}(\ID ,c\overline{\ID })$.

\begin{theorem}\label{sphetheo}
If $f(z)=\sum_{n=0}^{\infty} a_nz^n\in {\mathcal H}(\ID ,c\overline{\ID})$, then
\begin{equation}\label{sphe}
\chi \left(\sum_{n=0}^{\infty} |a_nz^n|,|a_0|\right)\leq \chi(a_0,\partial c\overline{\ID})
\end{equation}
for $|z|\leq 1/3$. Moreover, the bound $1/3$ is sharp.
\end{theorem}

Interestingly, if $f\in {\mathcal H}(\ID ,c\overline{\ID})$, then $f\prec \exp\circ W$ for some univalent function $W$ mapping $\ID $ onto the right-half plane. Thus $\exp\circ W\in {\mathcal H}(\ID ,c\overline{\ID})$ is a universal covering map. The proof of Theorem \ref{sphetheo} in \cite{Abu2} used the following key result.

\begin{lemma}
{\rm (see \cite{AbuHal-92})}\label{Hallen}
If $F$ is a univalent function mapping $\ID $ onto $\Omega$, where the complement of $\Omega$ is convex, and $F(z)\neq0$, then any analytic function
 $f\in S(F^n)$ for a fixed $n=1,2,\ldots$, can be expressed as
$$
f(z)=\int_{|x|=1} F^n(xz)\,d\mu(x)
$$
for some probability measure $\mu$ on the unit circle $|x|=1.$ (Here $S(F^n)$ is defined as in Section \ref{sec2.1}.)
Consequently,
$$ f(z)=\int_{|x|=1} \exp(F^n(xz))\,d\mu(x),
$$
for every $f\in S(\exp\circ F)$.
\end{lemma}

In the same paper \cite{Abu2}, a result under a more general setting than Theorem \ref{sphetheo} was also obtained.
\begin{theorem}
Let $\Delta$ be a compact convex body with $0\in \Delta$, $1\in \partial \Delta$. Suppose that there is a universal covering map
from $\ID$ into $c\Delta$ with a univalent logarithmic branch that maps $\ID $ into the complement of a convex set.
If $f(z)=\sum_{n=0}^{\infty} a_nz^n\in {\mathcal H}(\ID ,c\Delta)$ satisfies $a_0>1$, then inequality \eqref{sphe} holds for $|z|<3-2\sqrt{2}\approx 0.17157.$
\end{theorem}

Another paper by Abu Muhanna and Ali \cite{Abu4} considered the hyperbolic metric. Recall that \cite{Bea} the hyperbolic metric for $\ID$ is defined by
$$
\lambda_\ID (z)|dz|=\frac{2|dz|}{1-|z|^2},
$$
the hyperbolic length by
$$
L_{\ID} (\gamma) =\int_{\gamma} \lambda_\ID (z)\,|dz|,
$$
and the hyperbolic distance by
$$d_\ID(z,w)=\inf_{\gamma} L_\ID(\gamma)=\log \frac{1+\left|\frac{z-w}{1-z\overline{w}}\right|}{1-\left|\frac{z-w}{1-z\overline{w}}\right|}.
$$
Here the infimum is taken over all smooth curves $\gamma$ joining $z$ to $w$ in $\ID $. The function $\lambda_\ID (z)=2/(1-|z|^2)$ is known as the density of the hyperbolic metric on $\ID$. For any simply connected domain $\Omega$, $\lambda_{\Omega}$ can be computed via the formula
$$
\lambda_{\Omega}(w)=\frac{\lambda_\ID (f^{-1}(w))}{|f'(f^{-1}(w))|}, \quad w\in\Omega,
$$
where $f$ maps $\ID $ conformally onto $\Omega$. Note that the metric $\lambda_{\Omega}$ is independent of the choice of the conformal map $f$ used. The metrics $d_{\Omega}$ and $d_\ID $ satisfy the following relation:
\begin{theorem}
Let $f:\,\ID \rightarrow \Omega$ be analytic, where $\Omega$ is a simply connected subdomain of $\IC$. Then
$$
d_{\Omega} (f(z),f(w))\leq d_\ID (z,w).
$$
Equality is possible only when $f$ maps $\ID $ conformally onto $\Omega$.
\end{theorem}

The authors in \cite{Abu4} incorporated the hyperbolic metric into the Bohr inequality for three classes of functions.
Let $\mathbb{H}=\{z=x+iy:\,\text{Re}\, z>0\}$ be the right half-plane. Then
$$\lambda_{\mathbb{H}}|dz|=\frac{|dz|}{\text{Re}\,z}
$$
as shown in \cite[Example 7.2]{Bea}.

\begin{theorem}
{\rm (\cite[Theorem 2.1]{Abu4})}
Let $f(z)=\sum_{n=0}^{\infty} a_nz^n\in {\mathcal H}(\ID ,\mathbb{H})$. Then
$$
\sum_{n=1}^{\infty} |a_nz^n|\leq \frac{1}{\lambda_{\mathbb{H}}(a_0)}=\dist (a_0,\partial \mathbb{H})
$$
for $|z|\leq 1/3$. The bound is sharp.
\end{theorem}

The next result is on the class ${\mathcal H}(\ID ,\mathbb{P})$, where $\mathbb{P}=\{z:\,|\arg z|<\pi\}$ is a slit domain. Its hyperbolic metric \cite[Example 7.7]{Bea} is given by
$$
\lambda_{\mathbb{P}}|dz|=\frac{|dz|}{2|\sqrt{z}|\text{Re}\,\sqrt{z}}=\frac{|dz|}{2|z|\cos[(\arg z)/2]}\geq \frac{|dz|}{2|z|}.
$$

\begin{theorem}
{\rm (\cite[Theorem 2.3]{Abu4})}
Let $g(z)=\sum_{n=0}^{\infty} b_nz^n\in {\mathcal H}(\ID ,\mathbb{P})$. Then
$$
\sum_{n=1}^{\infty} |b_nz^n|\leq \frac{1}{2\lambda_{\mathbb{P}}(|b_0|)}=\dist (|b_0|,\partial \mathbb{P})
$$
for $|z|\leq 3-2\sqrt{2}\approx 0.17157$. The bound is sharp.
\end{theorem}

The final class treated was the class ${\mathcal H}(\ID ,c\overline{\ID })$. By the formula given earlier, it can be deduced that
$$
\lambda_{c\overline{\ID}}|dz|=\frac{|dz|}{|z|\log |z|}.
$$

\begin{theorem}
{\rm (\cite[Theorem 2.5]{Abu4})}
Let $c_0>1$ and $h(z)=\sum_{n=0}^{\infty} c_nz^n\in {\mathcal H}(\ID,c\overline{\ID })$. If $|z|\leq 1/3$, then
\begin{itemize}
\item[{\rm (a)}] $\ds \log \left(\sum_{n=0}^{\infty} |c_nz^n|\right)-\log c_0\leq \frac{1}{\lambda_{\mathbb{H}}(\log c_0)}=\dist (\log c_0,\partial \mathbb{H})$,
\item[{\rm (b)}] $\ds \sum_{n=1}^{\infty} |c_nz^n|\leq \frac{2}{\lambda_{c\overline{\ID}} (c_0)}$, provided $c_0\leq 2$.
\end{itemize}
\end{theorem}

\subsection{Bohr radius for concave-wedge domain}

In \cite{Abu3}, another class ${\mathcal H}(\ID,W_{\alpha})$ was considered, where
$$
W_{\alpha}:=\left\{w\in \mathbb{C}:\, |\arg w|<\frac{\alpha\pi}{2}\right\},\quad 1\leq\alpha\leq2,
$$
is a concave-wedge domain. In this instance, the conformal map of $\ID$ onto $W_{\alpha}$ is given by
\begin{equation}\label{F}
F_{\alpha,t}(z)=t\left(\frac{1+z}{1-z}\right)^{\alpha}=t\left(1+\sum_{n=1}^{\infty}A_nz^n\right), \quad t>0.
 \end{equation}
When $\alpha=1$, the domain reduces to a  convex half-plane, while the case $\alpha=2$ yields a slit domain.
The results are as follows:
\begin{theorem}\label{14theo1}
Let $\alpha\in[1,2]$. If $f(z)=a_0+\sum_{n=1}^{\infty} a_nz^n\in {\mathcal H}(\ID,W_{\alpha})$ with $a_0>0$, then
\begin{eqnarray*}
\sum_{n=1}^{\infty}|a_nz^n|\leq \dist (a_0,\partial W_{\alpha})
\end{eqnarray*}
for $|\,z|\leq r_{\alpha}=(2^{1/\alpha }-1)/(2^{1/\alpha }+1)$. The function $f=F_{\alpha,a_0}$ in \eqref{F} shows that the Bohr radius $r_{\alpha}$ is sharp.
\end{theorem}
\begin{theorem}\label{14theo2}
Let $\alpha\in[1,2]$. If $f(z)=a_0+\sum_{n=1}^{\infty} a_nz^n\in {\mathcal H}(\ID,W_{\alpha})$, then
\begin{eqnarray}\nonumber
\sum_{n=0}^{\infty}|a_nz^n|-|a_0|^*\leq \dist (|a_0|^*,\partial W_{\alpha})
\end{eqnarray}
for $|\,z|\leq r_{\alpha}=(2^{1/\alpha }-1)/(2^{1/\alpha }+1)$, where $|a_0|^*=F_{\alpha,1}\big (\big |F_{\alpha,1}^{-1}(a_0)\big|\big )$ and $F_{\alpha,1}$ is given by \eqref{F}. The function $f=F_{\alpha,|a_0|^*}$ shows that the Bohr radius $r_{\alpha}$ is sharp.
\end{theorem}

Theorem \ref{14theo1} is in the standard distance form for the Bohr theorem but under the condition $a_0>0.$ Theorem \ref{14theo2}, however, has no extra condition, but the inequality is only nearly Bohr-like. In a recent paper \cite{AliBarSoly}, Theorem \ref{14theo1} is shown to hold true without the assumption $a_0>0.$

\subsection{Bohr radius for a special subordination class}
The link between Bohr and differential subordination was also established in \cite{Abu3}. For $\alpha \geq \gamma \geq 0$, and for a given convex function $h\in {\mathcal H}(\ID)$, let
\begin{equation*}
R(\alpha,\gamma,h):=\{f\in {\mathcal H}(\ID):\,f(z)+\alpha z f'(z)+\gamma z^2f''(z)\prec h(z), \quad z\in \ID\}.
\end{equation*}
An easy exercise shows that 
$$
f(z)\prec q(z)\prec h(z), \quad \text{for all $f\in R(\alpha,\gamma,h)$},
$$
where
\begin{equation}\label{eq:q}
q(z)=\int_{0}^{1}\int_{0}^{1} h(zt^{\mu}s^{\nu}) \,dt\,ds\in R(\alpha,\gamma,h).
\end{equation}
Thus $R(\alpha,\gamma,h)\subset S(h)$ which is a subordination class.
\begin{theorem}
Let $f(z)=\sum_{n=0}^{\infty} a_nz^n \in R(\alpha,\gamma,h)$, and $h$ be convex. Then
\begin{eqnarray*}\label{eq:4}
\sum_{n=1}^{\infty} |a_nz^n| \leq \dist (h(0),\partial h(\ID))
\end{eqnarray*}
for all $|z|\leq r_{CV}(\alpha,\gamma)$, where $r_{CV}(\alpha,\gamma)$  is the smallest positive root of the equation
$$
\sum_{n=1}^{\infty}\frac{1}{(1+\mu n)(1+\nu n)}\,r^n=\frac{1}{2}.
$$
Further, this bound is sharp.  An extremal case occurs when $f(z)=q(z)$ as defined in \eqref{eq:q} and $h(z)=z/(1-z)$.
\end{theorem}

Analogously, if $h\in {\mathcal H}(\ID)$ is starlike, that is, $h$ is univalent in $\ID$ and the domain $h(\ID)$ is starlike with
respect to the origin, then the function $q$ in (\ref{eq:q}) is also starlike. A similar result for subordination to a
starlike function is readily obtained.

\begin{theorem}
Let $f(z)=\sum_{n=0}^{\infty} a_nz^n\in R(\alpha,\gamma,h)$, and $h$ be starlike. Then
\begin{eqnarray*}
\sum_{n=1}^{\infty} |a_nz^n|\leq \dist (h(0),\partial h(\ID))
\end{eqnarray*}
for all $|z|\leq r_{ST}(\alpha,\gamma)$, where $r_{ST}(\alpha,\gamma)$  is the smallest positive root of the equation
$$
\sum_{n=1}^{\infty}\frac{n}{(1+\mu n)(1+\nu n)}\,r^n=\frac{1}{4}.
$$
This bound is sharp. An extremal case occurs when $f(z)=q(z)$ as defined in \eqref{eq:q} and $h(z)=z/(1-z)^2$.
\end{theorem}

\subsection{Bohr's theorem for starlike logharmonic mappings}

We end Section 2 by discussing a recent extension of Bohr's theorem to the class of starlike logharmonic mappings. These mappings have been widely studied, for example, the works in  \cite{Ali,LiSW-13,MaoPW-13} and the references therein.

Let $\mathcal{B}(\ID)$ denote the set of all functions $a$ analytic in $\ID$ satisfying $|a(z)|<1$ in $\ID$. A logharmonic mapping defined in $\ID$ is a solution of the nonlinear elliptic partial differential equation
\begin{equation*}
\frac{\overline{f_{\overline{z}}}}{\overline{f}}=a\frac{f_{z}}{f},
\end{equation*}
where the second dilatation function $a$ lies in $\mathcal{B}(\ID)$. Thus the Jacobian
\begin{equation*}
J_{f}=\left\vert f_{z}\right\vert ^{2}(1-|a|^{2})
\end{equation*}
is positive and all non-constant logharmonic mappings are therefore sense-preserving and open in $\ID$.

If $f$ is a non-constant logharmonic mapping of $\ID $ which vanishes only at $z=0$, then $f$ admits the representation \cite{z1}
\begin{equation}  \label{eq1.2}
f(z)=z^{m}|z|^{2\beta m}h(z)\overline{g(z)},
\end{equation}
where $m$ is a nonnegative integer, $\mathrm{Re } \, \beta >-1/2$, and $h$ and $g$ are analytic functions in $\ID $ satisfying $g(0)=1$ and $h(0)\neq 0.$ The exponent $\beta $ in \eqref{eq1.2} depends only on $a(0)$ and can be expressed by
\begin{equation*}
\beta =\overline{a(0)}\frac{1+a(0)}{1-|a(0)|^{2}}.
\end{equation*}
Note that $f(0)\neq 0$ if and only if $m=0$, and that a univalent logharmonic mapping in $\ID$ vanishes at the origin if and only if $m=1$, that is, $f$ has the form
\begin{equation*}
f(z)=z|z|^{2\beta }h(z)\overline{g(z)},\quad z\in \ID ,
\end{equation*}
where $\mathrm{Re } \, \beta >-1/2$, $0\notin (hg)(\ID)$ and $g(0)=1$. In this case, it follows that $F(\zeta )=\log f(e^{\zeta })$ is a univalent harmonic mapping of the half-plane $\{\zeta :\,{\rm Re } (\zeta )<0\}$.

Denote by $S_{Lh}$ the class consisting of univalent logharmonic maps $f$ of the form
\begin{equation*}
f(z)=zh(z)\overline{g(z)}
\end{equation*}
with the normalization $h(0)=g(0)=1.$ Also denote by $ST^0_{Lh}$ the class consisting of functions $f\in S_{Lh}$ which maps $\ID$ onto a starlike domain (with respect to the origin). Further let $S^{\ast}$ be the usual class of normalized analytic functions $f$ satisfying $f(\ID)$ is a starlike domain.

\begin{lemma}\cite{z5} \label{loglemma}
Let $f(z)=zh(z)\overline{g(z)}$ be logharmonic in $\ID$. Then $f\in ST^0_{Lh}$ if and only if $\varphi(z)=zh(z)/g(z)\in S^{\ast}$.
\end{lemma}

This lemma shows the connection between starlike logharmonic functions and starlike analytic functions. The authors made use of Lemma \ref{loglemma} to obtain necessary and sufficient conditions on $h$ and $g$ so that the function $f(z)=zh(z)\overline{g(z)}$ belongs to $ST^0_{Lh}$
(see \cite[Theorem 1]{Ali}), from which resulted in a sharp distortion theorem.
\begin{theorem}\label{distort}
Let $f(z)=zh(z)\overline{g(z)}\in ST^0_{Lh}$. Then
\begin{equation*}
\frac{1}{1+|z|}\exp \left(\frac{-2|z|}{1+|z|}\right)\leq |h(z)|\leq \frac{1}{1-|z|}%
\exp \left(\frac{2|z|}{1-|z|}\right),
\end{equation*}
\begin{equation*}
(1+|z|)\exp \left(\frac{-2|z|}{1+|z|}\right)\leq |g(z)|\leq (1-|z|)\exp \left(\frac{2|z|}{1-|z|}\right),
\end{equation*}
and
\begin{equation*}
|z|\exp \left( \frac{-4|z|}{1+|z|}\right) \leq |f(z)|\leq |z|\exp \left(
\frac{4|z|}{1-|z|}\right) .
\end{equation*}

Equalities occur if and only if $h, g, $ and $f$ are, respectively, appropriate rotations of $h_0,g_0$ and $f_0$, where
\begin{equation*}
h_{0}(z)=\frac{1}{1-z}\exp\left( \frac{2z}{1-z}\right)=\exp\left(\sum_{n=1}^{\infty} \left(2+\frac{1}{n}\right)z^n\right),
\end{equation*}
\begin{equation*}
g_{0}(z)=(1-z)\exp \left(\frac{2z}{1-z}\right)=\exp\left(\sum_{n=1}^{\infty} \left(2-\frac{1}{n}\right)z^n\right),
\end{equation*}
and
\begin{equation*}
f_0(z)=zh_0(z)\overline{g_0(z)}=\frac{z(1-\overline{z})}{1-z}\exp\left( \text{Re} \left(\frac{4z}{1-z}\right)\right).
\end{equation*}
\end{theorem}
The function $f_0$ is the logharmonic Koebe function. Theorem \ref{distort} then gives

\begin{corollary}\label{bound}
Let $f(z)=zh(z)\overline{g(z)}\in ST^0_{Lh}$. Also, let $H(z)=zh(z)$ and $G(z)=zg(z)$. Then
\begin{equation*}
\frac{1}{2e} \leq \dist (0,\partial H(\ID))\leq 1,
 \quad
\frac{2}{e} \leq \dist (0,\partial G(\ID))\leq 1,
\end{equation*}
and
\begin{equation*}
\frac{1}{e^{2}}\leq \dist (0,\partial f(\ID))\leq 1.
\end{equation*}
Equalities occur if and only if $h, g$ and $f$ are, respectively, suitable rotations of $h_0,g_0$ and $f_0$.
\end{corollary}

Finally, with the help of Corollary \ref{bound} and the sharp coefficient bounds from \cite[Theorem 3.3]{Periodic}, the Bohr theorems are obtained.
\begin{theorem}
Let $f(z)=zh(z)\overline{g(z)}$ $\in ST^0_{Lh},$ $H(z)=zh(z) $ and $G(z)=zg(z)$. Then
\begin{itemize}
  \item [{\rm (a)}] $\ds |z|\exp\left(\overset{\infty }{\underset{n=1}{\sum }}|a_{n}||z|^{n}\right)\ \leq \dist (0,\partial
H(\ID))$
for $|z|\leq r_H\approx 0.1222,$ where $r_H$ is the unique root in $(0,1)$ of
\begin{equation*}
\frac{r}{1-r}\exp\left(\frac{2r}{1-r}\right)=\frac{1}{2e},
\end{equation*}
  \item [{\rm (b)}] $\ds
|z|\exp\left(\overset{\infty }{\underset{n=1}{\sum }}|b_{n}||z|^{n}\right)\ \leq \dist (0,\partial G(\ID))$
for $|z|\leq r_G\approx 0.3659$, where $r_G$ is the unique root in $(0,1)$ of
\begin{equation*}
r(1-r)\exp\left(\frac{2r}{1-r}\right)=\frac{2}{e}.
\end{equation*}
\end{itemize}
Both radii are sharp and are attained, respectively, by appropriate rotations of $H_0(z)=zh_0(z)$ and $G_0(z)=zg_0(z)$. Here $h_0$ and $g_0$ 
are given in Theorem \ref{distort}.
\end{theorem}

\begin{theorem}
Let $f(z)=zh(z)\overline{g(z)}\in ST^0_{Lh}$. Then, for any real $t$,
\begin{equation*}
|z|\exp\left(\overset{\infty }{\underset{n=1}{\sum }} \left|a_n+e^{it}b_n\right||z|^{n}\right)\leq \dist (0,\partial f(\ID))
\end{equation*}
for $|z|\leq r_0\approx 0.09078$, where $r_0$ is the unique root in $(0,1)$ of
$$
r\exp\left(\frac{4r}{1-r}\right)=\frac{1}{e^2}.
$$
The bound is sharp and is attained by a suitable rotation of the logharmonic Koebe function $f_0$.
\end{theorem}

\section{Dirichlet Series and $n$-dimensional Bohr radius}\label{sec1-ndim}

\subsection{Bohr and the Dirichlet Series}

In \cite{BaluCQ-2006}, Balasubramanian {\it et al.}  extended the Bohr inequality to the setting of Dirichlet series.
This paper brings the Bohr phenomenon back to its origins since Bohr radius for
power series on the disk originated from studying problems on absolute convergence \cite{Bohr-14} in the theory of
Dirichlet series.

For $1 \leq p < \infty$, let $D^p$ be the space of ordinary Dirichlet series consisting of
$f(s)=\sum_{n=1}^\infty a_n n^{-s}$ in $\mathbb{H}=\{s=\sigma +it:\, \sigma >0\}$ corresponding to the Hardy space of order $p$.
The space $D^p$ is the completion of the space of Dirichlet polynomials $P(s) = \sum_{n=1}^N a_n n^{-s}$ in
the norm
$$\|P\| = \left (\lim_{T \to \infty} {1 \over 2T} \int_{-T}^T |P(it)|^p \,dt\right )^{1/p},
$$
which is equivalent to requiring $\sum|a_n|^2<\infty$ when $p=2$. The space $D^\infty$ consists of the space of Dirichlet series as above with $\|f\|_\infty := \sup \{|f(s)|:\, \sigma ={\rm Re}\, s>0 \}< \infty$.
Then the Bohnenblust-Hille theorem \cite{BonHille-1931} takes the form

\begin{theorem}\label{BonHille-Theo}
The infimum of $\rho$ such that $\sum_{n=1}^\infty|a_n|n^{-\rho}<\infty$ for every $\sum_{n=1}^\infty a_nn^{-s}$ in $D^\infty$ equals $1/2$.
\end{theorem}

For $k \geq 1$,  let $D_k^{\infty}$ denote the subspace of $D^\infty$ consisting of $f(s) = \sum_{n=1}^\infty a_n n^{-s}$ such that
$a_n = 0$  whenever the number of prime divisors of $n$ exceeds $k$.
If $(E,\| \cdot \|)$ is a Banach space of Dirichlet series, the isometric Bohr abscissa and the isomorphic Bohr abscissa are,
respectively, defined as
$$\rho_1(E) = \min \left\{\sigma \geq 0:\, \sum_{n=1}^\infty |a_n| n^{-\sigma} \leq \|f\| ~\mbox{ for all $f \in E$}\right \},
$$
and
$$\rho (E) = \inf \left \{\sigma \geq 0:\, \exists ~ C_\sigma \in (0,\infty )~
\mbox{ such that $\sum_{n=1}^\infty |a_n\| n^{-\sigma} \leq C_\sigma \|f\|$  for all $f \in E$}\right \}.
$$
By using a number of recent developments in this topic (some of them related to the hypercontractivity properties of the Poisson kernel),
the authors in \cite{BaluCQ-2006} obtained among others the following results:

\begin{enumerate}
\item[(1)] If $1 \leq p < \infty$, then $\rho (D^p) = 1/2$, but this value is not attained. For $p=\infty$,
this is equivalent to determining the maximum possible width of the strip of uniform, but not absolute, convergence of Dirichlet series
(see Theorem \ref{BonHille-Theo}).
\item[(2)] $\rho (D^\infty ) = 1/2$, and this value is attained. So $\sum_{n=1}^\infty \vert a_n\vert n^{-1/2} \leq C \Vert f\Vert _\infty$ for some absolute constant $C$ (see Theorem \ref{BonHille-Theo}).
\item[(3)] Let $p \in [0,1]$. Every $f(s) = \sum_{n=1}^\infty a_n n^{-s} \in D^\infty$ satisfies $\sum_{n=1}^\infty \vert a_n n^{-\sigma}\vert ^p<\infty$ whenever $\sigma \geq \sigma_0 := 1/p - 1/2$. If $\sigma < \sigma_0$, there is $f \in D^\infty$ such that the last sum is infinite.
\item[(4)] $\rho (D_k^\infty ) = 1/2 - 1/(2k)$, and it is attained.
 \item[(5)] $\rho_1 (D_1^\infty ) = 0$.
\item[(6)] $1.5903 < \rho_1 (D_2^\infty ) < 1.5904$.
\item[(7)] $1.585 < \log 3 / \log 2 \leq \rho_1 (D^\infty ) \leq 1.8154$. In particular, $\sum_{n=1}^\infty \vert a_n\vert n^{-2} \leq \Vert f\Vert _\infty$.
\end{enumerate}

\subsection{The $n$-dimensional Bohr radius}
Mathematicians have studied various generalizations of Bohr theorem, for example, in the works of \cite{Aizen-00-1,Aizen-05-3,Boa,BoasKhavin-97-4,DIN,DIN91,HH,HHK}
and the references therein. One generalization uses power series representation of holomorphic functions defined on a complete Reinhardt domain, that is, a bounded complete $n$-circular domain in ${\mathbb C}^n$. In order to present and summarize certain multidimensional analogs of the Bohr and related inequalities, consider an $n$-variable power series
$\sum_{\alpha}c_{\alpha}z^{\alpha}$ in the standard multi-index notation, where $\alpha$ denotes an $n$-tuple $(\alpha_1, \ldots, \alpha_n)$
of nonnegative integers, $|\alpha|$ denotes the sum $\alpha_1 + \cdots + \alpha_n$ of its components,
$\alpha !$ denotes the product $\alpha_1! \cdots  \alpha_n!$ of the factorials of its components,
$z$ denotes the $n$-tuple $(z_1, \ldots, z_n)$ of complex numbers, and  $z^{\alpha}=z_1^{\alpha _1} \cdots z_n^{\alpha _n}$.

Let $D$ be a complete Reinhardt domain. Denote by $R(D)$ the largest nonnegative number $r$ such that whenever  the power series $\sum_{\alpha}c_{\alpha}z^{\alpha}$ converges in $D$ with
$\left |\sum_{\alpha}c_{\alpha}z^{\alpha}\right |< 1,$ then $\sum_{\alpha}\left |c_{\alpha}z^{\alpha}\right |<1$
in the homothety $rD$. The number $R(D)$ is called the Bohr radius.
In the case of $n$-dimensional unit polydisk $\ID^n=\{(z_1, \ldots, z_n):\, \max _{1\leq j\leq n} |z_j|<1 \}$,
Boas and Khavinson in \cite[Theorem 2]{BoasKhavin-97-4} showed the following estimate for the
Bohr radius $R(D)$ (also known as the {\it first Bohr radius}, and denoted by $K_n$):
$$
\frac{1}{3\sqrt{n}}\leq K_n\leq 2\sqrt{\frac{\log n}{n}}.
$$
Note that the radius decreases to zero as the dimension of the domain increases.  The result of Boas and Khavinson
stimulated a lot of interest in Bohr type questions
and has brought Bohr theorem to prominence even though the generalization of Bohr radius to the unit polydisk in $\mathbb{C}^n$ was first
studied in \cite{DIN}. By using the fact that the Bohnenblust-Hille inequality is hypercontractive, Defant \emph{{\it et al.}} \cite{DFOOS} obtained the optimal asymptotic estimate for this radius to be
\begin{equation*}
K_n=b(n)\sqrt{\frac{\log n}{n}},
\end{equation*}
where $1/\sqrt{2}+o(1)\leq b(n)\leq 2$. Bayart et al. \cite{Bay} proved that $K_n$ behaves asymptotically as $\sqrt{(\log n)/n}$ and further improved the bounds for $K_n$ to
\begin{equation}\label{ndimB}
K_n=c(n)\sqrt{\frac{\log n}{n}},
\end{equation}
where $1+o(1)\leq c(n)\leq 2$. The article of Bohnenblust-Hille remains a seminal contribution, and the hypercontractive polynomial Bohnenblust–Hille inequality
is the best one can hope for. It has several interesting consequences, and leads to precise asymptotic results regarding
certain Sidon sets, Bohr radii for polydisks, and the moduli of the coefficients of functions in $H^\infty$.

A few years after the first appearance of $K_n$ in \cite{BoasKhavin-97-4}, Boas \cite{Boa} extended the Bohr theorem to the
complex Banach space $\ell_p^n$ whose norm is defined by
$$\|z\|_{\ell_p^n}:=\left(\sum_{j=1}^n |z_j|^p\right)^{1/p}.
$$
When $p=\infty$, the unit ball is to be interpreted as the unit polydisk in $\mathbb{C}^n$. Denote by $K(B_{\ell_p^n})$ the corresponding Bohr radius.
For $1\leq p\leq \infty$, it was shown in \cite{Boa} that
$$
\frac{1}{c} \left(\frac{1}{n}\right)^{1-\frac{1}{\min\{p,2\}}}\leq K(B_{\ell_p^n}) \leq c\left(\frac{\log n}{n}\right)^{1-\frac{1}{\min\{p,2\}}},
$$
where $c>0$ is a constant independent of $p,n$. The lower bound was then improved to the value $\sqrt{(\log n/ \log \log n)/n}$ in \cite{DF}. On the other hand, Aizenberg  \cite{Aizen-00-1} proved that
$$
\frac{1}{3e^{1/3}}\leq K(B_{\ell_1^n}) \leq 1/3,
$$
where $1/3$ is the best upper bound.  Also note that this estimate does not depend on $n$. In the same paper, Aizenberg defined the {\it second Bohr radius} $B_n(G),$ which is the largest radius $r$ such that whenever a multidimensional power series $\sum_{\gamma} a_{\gamma} z^{\gamma}$ is bounded by 1 in the complete Reinhardt domain $G$, then $\sum_{\gamma} \sup_{rG} |a_{\gamma} z^{\gamma}|\leq1$. General lower and upper estimates for the first and the second Bohr radii of
bounded complete Reinhardt domains are given in \cite{DGM2}. Results from both papers \cite{DF} and \cite{DGM2} were proved using certain theorems from \cite{DGM}, which was the first paper linking multidimensional Bohr study to local Banach space theory. The estimates for $K_n$ obtained in \cite{DGM} were in terms of unconditional basis constants and Banach-Mazur distances. We refer to \cite{DP} for a survey on these studies.

\subsection{Bohr radius in the study of Banach spaces}
A new Bohr-type radius can be found in \cite{DMS} which relates to the study of Banach spaces. Let $v:\, X\rightarrow Y$ be a
bounded operator between complex Banach spaces, $n\in \mathbb{N},$ and $\lambda\geq \|v\|$. The {\it $\lambda$-Bohr radius of $v$},
denoted by $K_n(v,\lambda)$, is the supremum of all $r\geq 0$ such that for all holomorphic functions
$f(z)=\sum_{\alpha\in \mathbb{N}_0^n} c_{\alpha}z^{\alpha}$ on the $n$-dimensional unit polydisk $\ID^n$,
$$
\sup_{z\in r\ID^n} \sum_{\alpha\in \mathbb{N}_0^n} \left \|v(c_{\alpha})z^{\alpha}\right \|_{Y}\leq \lambda \sup_{z\in \ID^n}
\Big \|\sum_{\alpha\in \mathbb{N}_0^n} c_{\alpha}z^{\alpha}\Big \|_{X}.
$$
If $X=\mathbb{C}$, $v$ is the identity on $X$ and $\lambda=1$, then $K_n(v,\lambda)=K_n$ is exactly the $n$-dimensional Bohr radius previously defined.
Thus, the main goal in \cite{DMS} was to study the Bohr radii of the $n$-dimensional unit polydisk for holomorphic functions defined on $\ID^n$ with values
in Banach spaces, for example, by obtaining upper and lower estimates for Bohr radii $K_n(v,\lambda)$ of specific operators $v$ between Banach spaces.
Interestingly, Dixon \cite{Dixon-95-7} paid attention to applications of Bohr phenomena to
operator theory showing that Bohr theorem is useful in the characterization
of Banach algebras that satisfy von Neuman's inequality. Later Paulsen {\it et al.} in \cite{PaulPopeSingh-02-10} continued
the work in this line of investigation.


It was shown in \cite{BenDahKha} that the analogous Bohr theorem fails in the Hardy spaces $H^q$, $0<q<\infty$, equipped with the corresponding Hardy norm. The authors also showed how renorming a space affected the Bohr radius. In \cite{AbuGun}, Abu-Muhanna and Gunatillake found the Bohr radius for the weighted Hardy Hilbert spaces, and again showed that no Bohr radius exists for the classical Hardy space $H^2$. The Bohr research on multidimensional weighted Hardy-Hilbert was earlier done in \cite{PaulPopeSingh-02-10}. Thus, in view of the supremum norm, we can say that for the standard basis $(z^n)_{n=0}^{\infty}$, there exists a compact set $\{z\in \mathbb{C}:\,|z|\leq 1/3\}$ which lies in the open set $U\subset \mathbb{C}$ such that the norm version Bohr inequality occurred. Hence $(z^n)_{n=0}^{\infty}$ is said to have the Bohr property.

In \cite{AAD2} and \cite{AAD}, Aizenberg {\it et al.} considered the general bases for the space of holomorphic functions ${\mathcal H}(M)$ on
a complex manifold $M$. A basis $(\phi_n)_{n=0}^{\infty}$ in ${\mathcal H}(M)$ is said to have the {\it Bohr Property} (\textbf{BP}) if
there exist an open set $U\subset M$ and a compact set $K\subset M$ satisfying
\begin{equation*}
\sum|c_n|\sup_{U} |\phi_n(z)|\leq \sup_{K} |f(z)|
\end{equation*}
for all $f=\sum c_n\phi_n\in {\mathcal H}(M)$. It was shown in \cite{AAD} that when $\phi_0=1$ and $\phi_n(z_0)$, $n\geq1$, vanishes for some $z_0\in M$, then $(\phi_n)_n^{\infty}$ has the \textbf{BP}. A generalization of this result can be found in \cite[Theorem 4]{AAD2}. Aytuna and Djakov in \cite{Ayt} introduced the term Global Bohr Property: a basis $(\phi_n)_n^{\infty}$ in the space of entire functions ${\mathcal H}(\mathbb{C}^n)$ has the {\it Global Bohr Property} (\textbf{GBP}) if for every compact $K\subset \mathbb{C}^n$ there is a compact $K_1\supset K$ such that
$$
\sum |c_n|\sup_{K}|\phi_n(z)|\leq \sup_{K_1} |f(z)|
$$
for all $f=\sum c_n\phi_n\in {\mathcal H}(\mathbb{C}^n)$. They showed that a basis $(\phi_n)_n^{\infty}$ has the \textbf{GBP} if and only if one of the functions $\phi_n$ is a constant. As pointed out in \cite{Ayt}, the inequality of \textbf{GBP} was first stated in a paper by Lass\`{e}re and Mazzilli \cite{Las}. They studied the power series expressed in Faber polynomial basis which is associated with a compact continuum in $\mathbb{C}$. In fact, the relation between
\textbf{BP} and Faber polynomials was first discovered by Kaptano\v{g}lu and Sad\i k \cite{Kap2}. The radius obtained in \cite{Kap2} was not sharp. It was later solved in \cite{Las2} by using better elliptic Carath\'{e}odory's inequalities.

The connection between Hadamard real part theorem and Bohr theorem can be seen in \cite{KM}, in which Kresin and Maz'ya introduced the Bohr type real part estimates and proved the following theorem by applying the $\ell_p$ norm on the remainder of the power series expansion:
\begin{theorem}
Let $f(z)=\sum_{n=0}^{\infty} a_nz^n\in \mathcal{A}$ with
$$\sup_{|\zeta|<1} {\rm Re}\,(e^{-i\arg f(0)}f(\zeta))<\infty,
$$
where $\arg f(0)$ is replaced by zero if $f(0)=0.$ Then for any $q\in (0,\infty]$, integer $m\geq1$, and $|z|\leq r_{m,q},$ the inequality
$$\left(\sum_{n=m}^{\infty} |a_nz^n|^q \right)^{1/q}\leq \sup_{|\zeta|<1}  {\rm Re}\,(e^{-i\arg f(0)}f(\zeta))-|f(0)|
$$
holds, where $r_{m,q}\in (0,1)$ is the root of the equation $2^qr^{mq}+r^q-1=0$ if $0<q<\infty$, and $r_{m,\infty}:=2^{-1/m}$.
The radius $r_{m,q}$ is best possible.
\end{theorem}

With $(q,m)=(1,1)$, the result reduces to the sharp inequality obtained by Sidon \cite{Sidon-27-15} and
which contains the classical Bohr inequality. For a discussion on the Bohr-type real part estimates, we refer to \cite[Chapter 6]{KM2}.

There are still many possible directions of extending the Bohr theorem. For example, in \cite{Kap2}, the authors considered
the domain of functions bounded by ellipse instead of the unit disk $\mathbb{D}$. However, the Bohr radius does not exist for the space of holomorphic functions in an annulus equipped with the natural basis \cite{Aizen-12}. The Bohr radius for the class of analytic functions defined on
$$\{z:\,|z+\gamma/(1-\gamma)|<1/(1-\gamma)\},~~ 0\leq \gamma<1,
$$
was given in \cite{Fou2}. Liu and Wang \cite{Liu} proved another kind of extension of the classical Bohr inequality involving bounded symmetric domains.

\begin{theorem}
Let $\Omega$ denote one of the four classical domains in the sense of Hua \cite{Hua} or the unit polydisk in $\mathbb{C}^n$. Denote by $\| \cdot \|_{\Omega}$ the Minkowski norm associated to $\Omega$. Let $f:\,\Omega\rightarrow \Omega$ be a holomorphic map with
$$f(z)=\sum_{k=0}^{\infty} f_k(z)
$$
as its Taylor expansion in $k$-homogeneous polynomials $f_k$. Let $\phi\in {\rm Aut}\, \Omega$ such that $\phi(f(0))=0$. Then
$$\sum_{k=0}^{\infty} \frac{\|D\phi(f(0))\cdot f_k(z)\|_{\Omega}}{\|D\phi(f(0))\|_{\Omega}}<1
$$
for all $z\in \Omega$ satisfying $\|z\|_{\Omega}<1/3$.
\end{theorem}

Using a different approach, Roos \cite{Roos} extended the theorem to any bounded circled symmetric domain. Earlier results on the generalization of Bohr theorem using homogeneous expansions can also be found in \cite[Theorem 8]{Aizen-00-1} and \cite{Aizen-05-3}. Meanwhile, the generalization of both the results \cite{Liu} and \cite[Theorem 8]{Aizen-00-1} was obtained by Hamada {\it et al.} \cite{HHK}.

On the other hand, Guadarrama \cite{Gua} considered the polynomial Bohr radius defined by
$$ R_n=\sup_{p\in \mathcal{P}_n}\left \{r\in (0,1):\,\sum_{k=0}^n |a_k|r^k\leq \|p\|_{\infty},\quad p(z)=\sum_{k=0}^n a_kz^k\right \},
$$
where $\mathcal{P}_n$ consists of all the complex polynomials of degree at most $n$. The author showed that
$$
C_1\frac{1}{3^{n/2}}\leq R_n-1/3\leq C_2\frac{\log n}{n}
$$
for some positive constants $C_1$ and $C_2$. Subsequently, Fournier \cite{Fou} computed and obtained an explicit formula for
$R_n$ by applying the notion of bounded-preserving operators. The following result concerning the asymptotic behaviour of $R_n$ was proved only recently in \cite{Chu}:
$$
\lim_{n\rightarrow \infty} n^2\left(R_n-\frac{1}{3}\right)=\frac{\pi^2}{3}.
$$
As remarked earlier,  the authors in  \cite{PaulPopeSingh-02-10} square the constant term in the expansion of $f$  and obtained the sharp Bohr radius $1/2$.
A similar idea was adopted by Blasco in \cite{OB}. The author introduced and studied the radius
$$R_{p,q}(X)=\inf\{R_{p,q}(f,X):\,\sup_{|z|<1} \|f\|_X\leq 1\},
$$
where $X=L_p(\mu)$ or $X=\ell_p$ spaces and
$$R_{p,q}(f,X)=\sup\left \{r\geq0:\,\|a_0\|^p_X+\left (\sum_{n=1}^{\infty} \|a_n\|_X r^n\right )^q\leq1, ~ f(z)=\sum_{n=0}^{\infty} a_nz^n\right \}.
$$
Popescu, {\it et al.}  \cite{Pop,PaulPopeSingh-02-10,PaulSingh-06-12} established the operator-theoretic Bohr radius.
In \cite{Kap}, Kaptano\v{g}lu studied the Bohr phenomenon for elliptic equations by considering the case of harmonic
functions for the Laplace-Beltrami operator. The Bohr radii for classes of harmonic, separately harmonic and
pluriharmonic functions were evaluated in \cite{AizenTark-01-2}. Extension of Bohr theorem to uniform algebra
can also be found in \cite{PaulSingh-04-11}.

%

\subsection{Concluding remarks on multidimensional Bohr radius}

To conclude, we discuss some results with regard to the multidimensional Bohr radius. For functions $f$ holomorphic in $\mathbb{D}^n$ of the form
\begin{equation}\label{ndim}
f(z)=\sum_{\alpha}c_{\alpha}z^{\alpha}=\sum_{k=0}^{\infty} \sum_{|\alpha|=k}c_{\alpha}z^{\alpha},
\end{equation}
its associated majorant series is given by
$$
M_f(z)=\sum_{\alpha}|c_{\alpha}z^{\alpha}|=\sum_{k=0}^{\infty} \sum_{|\alpha|=k}|c_{\alpha}z^{\alpha}|.
$$
Also for an integer $m\ge 1$, we extend the definition of an $m$-symmetric analytic function of single variable to $n$-variable: a function $f$ holomorphic  in $\mathbb{D}^n$ is called $m$-symmetric if $f(e^{2\pi i/m}z)=f(z)$ for all $z\in \mathbb{D}^n$. Also note that a holomorphic  function $f$ of the form \eqref{ndim} is $m$-symmetric if and only if its Taylor expansion has the $m$-symmetric form
$$
f(z)=\sum_{k=0}^{\infty} \sum_{|\alpha|=k}c_{m\alpha}z^{m\alpha},
$$
where $m\alpha=(m\alpha_1,\ldots,m\alpha_n)$. Therefore if $f(z)=\sum_{k=0}^{\infty}\sum_{|\alpha|=k}c_{m\alpha}z^{m\alpha}$ is holomorphic in $\mathbb{D}^n$, then by letting $\zeta=z^m$, it follows that the function $g(\zeta)=\sum_{k=0}^{\infty}\sum_{|\alpha|=k}c_{m\alpha}\zeta^{\alpha}$  is also holomorphic  in $\mathbb{D}^n$. Hence the following result is obtained as a consequence of the $n$-dimensional Bohr theorem \eqref{ndimB}.

\begin{theorem}\label{multisym}
If $f(z)=\sum_{k=0}^{\infty} \sum_{|\alpha|=k}c_{m\alpha}z^{m\alpha}$ is holomorphic in $\mathbb{D}^n$ for some integer $m\ge 1$, and $\left|f(z)\right|<1$ in $\mathbb{D}^n$, then $M_f(z)<1$ holds in $K_{n,m}\cdot \mathbb{D}^n$ with
\begin{equation*}
K_{n,m}=\sqrt[m]{c(n)\sqrt{\frac{\log n}{n}}},
\end{equation*}
and $1+o(1)\leq c(n)\leq 2$.
\end{theorem}

Similar to the case of single variable, for a function $f$ holomorphic in $\mathbb{D}^n$ of the form \eqref{ndim}, its alternating series can be defined as
$$
A_f(z)=\sum_{k=0}^{\infty}(-1)^k \sum_{|\alpha|=k}|c_{\alpha}z^{\alpha}|.
$$
Adopting the idea from \cite{Aizen-00-1}, denote by $B_{A,n}$ the largest number $r$ such that
$$
M_rA_{f}(z)=\sum_{k=0}^{\infty}(-1)^k \sum_{|\alpha|=k}\sup_{\mathbb{D}_r^n}|c_{\alpha}z^{\alpha}|<1,
$$
where $r>0$ and $\mathbb{D}_r^n=r\cdot \mathbb{D}^n$ is the homothetic transformation of $\mathbb{D}^n$. The following majorant-type series
$$
M_{r,f}(z)=\sum_{k=0}^{\infty} \sum_{|\alpha|=k} \sup_{\mathbb{D}_r^n}|c_{\alpha}z^{\alpha}|
$$
will be required in the sequel.

\begin{theorem}
If $f(z)=\sum_{k=0}^{\infty} \sum_{|\alpha|=k}c_{\alpha}z^{\alpha}$ is holomorphic in $\mathbb{D}^n$ and $\left|f(z)\right|<1$ in $\mathbb{D}^n$, then $|M_rA_f(z)|<1$ holds in $B_{A,n}\cdot \mathbb{D}^n,$ where
$1-\sqrt[n]{2/5} \leq B_{A,n}.$
\end{theorem}

\begin{proof}
If $z=(z_1,\ldots,z_n)\in \mathbb{D}^n$, then $-z=(-z_1,\ldots,-z_n)$ and
$$
(-z)^{\alpha}=(-z_1)^{\alpha_1}\cdots(-z_n)^{\alpha_n}=(-1)^{|\alpha|}z^{\alpha}.
$$
Define the even and odd parts of $f$ to, respectively, be
$$f_e(z)=\frac{1}{2}(f(z)+f(-z))=\sum_{k=0}^{\infty} \sum_{|\alpha|=2k}c_{\alpha}z^{\alpha},
$$
and
$$f_o(z)=\frac{1}{2}(f(z)-f(-z))=\sum_{k=0}^{\infty} \sum_{|\alpha|=2k+1}c_{\alpha}z^{\alpha}.
$$
As $\mathbb{D}$ is convex, it follows that $\left|f_e(z)\right|<1$ and $\left|f_o(z)\right|<1$ in $\mathbb{D}^n$.

Now, Wiener method (see the proof of \cite[Theorem 2]{BoasKhavin-97-4}) and the multidimensional Cauchy estimate yield
$$
|c_{\alpha}|\leq 1-|c_0|^2 ~\mbox{ for }~|\alpha|\geq 1.
$$
The inequality then gives
\begin{align*}
M_rA_f(z)&=M_{r,f_e}(z)-M_{r,f_o}(z)\leq M_{r,f_e}(z)\\
&=\sum_{k=0}^{\infty} \sum_{|\alpha|=2k} \sup_{r\cdot\mathbb{D}^n}|c_{\alpha}z^{\alpha}|=\sum_{k=0}^{\infty} \sum_{|\alpha|=2k} |c_{\alpha}| \sup_{r\cdot\mathbb{D}^n}|z^{\alpha}|\\
&\leq |c_0|+(1-|c_0|^2)\sum_{k=1}^{\infty} r^{2k}\sum_{\alpha_1+\cdots+\alpha_n=2k} 1\\
&=|c_0|+(1-|c_0|^2)\sum_{k=1}^{\infty}  \binom{2k+n-1}{2k} r^{2k}.
\end{align*}
Since $(1-r)^{-n}=\sum_{k=0}^{\infty} \binom{n+k-1}{k} r^k$ and $|c_0|<1$, it follows that
\begin{align*}
M_rA_f(z)<|c_0|+(1-|c_0|)\left(\frac{1}{(1-r)^n}+\frac{1}{(1+r)^n}-2\right).
\end{align*}
Thus $M_rA_f(z)<1$ when
\begin{equation}\label{alt1}
\frac{1}{(1-r)^n}+\frac{1}{(1+r)^n}\leq 3.
\end{equation}
On the other hand,
\begin{align*}
M_rA_f(z)&=M_{r,f_e}(z)-M_{r,f_o}(z)\geq -M_{r,f_o}(z)\\
&=-\sum_{k=0}^{\infty} \sum_{|\alpha|=2k+1} \sup_{r\cdot\mathbb{D}^n}|c_{\alpha}z^{\alpha}|=-\sum_{k=0}^{\infty} \sum_{|\alpha|=2k+1} |c_{\alpha}|\sup_{r\cdot\mathbb{D}^n}|z^{\alpha}|\\
&\geq -(1-|c_0|^2)\sum_{k=0}^{\infty} r^{2k+1}\sum_{\alpha_1+\cdots+\alpha_n=2k+1} 1\\
&>-\sum_{k=0}^{\infty}  \binom{n+2k}{2k+1} r^{2k+1}=\frac{1}{2(1+r)^n}-\frac{1}{2(1-r)^n}.
\end{align*}
Thus $M_rA_f(z)>-1$ when
\begin{equation}\label{alt2}
\frac{1}{(1-r)^n}-\frac{1}{(1+r)^n}\leq 2.
\end{equation}
Adding \eqref{alt1} and \eqref{alt2} gives
$ r\leq 1-\sqrt[n]{2/5}.
$
\qed\end{proof}

\noindent \textbf{Acknowledgment.} This work benefited greatly from the stimulating discussions and careful scrutiny of
the manuscript by our student, Zhen Chuan Ng. The work of the second author was supported in parts by a research university
grant from Universiti Sains Malaysia. The research of the third author was supported by project
RUS/RFBR/P-163 under the Department of Science \& Technology, India.

\bibliographystyle{amsplain}

\end{document}